\documentclass[11pt,a4paper]{amsart}

\usepackage{fancyhdr} 
\usepackage{stmaryrd} 
\usepackage{combelow} 
\usepackage{hyperref} 
\usepackage{indentfirst} 
\usepackage{titlesec} 
\usepackage[a4paper]{geometry} 
\usepackage{xstring,xifthen} 
\usepackage{xcolor} 
\usepackage{graphicx} 
\usepackage{float} 
\usepackage{subcaption}

\usepackage{yhmath} 
\usepackage{etoolbox}
\usepackage{array}
\usepackage{tabularx}
\usepackage{siunitx}
\usepackage{tikz}
\usepackage{enumitem}
\usepackage{tikz-cd}

\setlist[enumerate]{itemsep=5pt}
\setlist[itemize]{itemsep=5pt}

\newcolumntype{Y}{>{\centering\arraybackslash}X}

\AtBeginEnvironment{theorem}{\smallskip}
\AtBeginEnvironment{proposition}{\smallskip}
\AtBeginEnvironment{definition}{\smallskip}
\AtBeginEnvironment{subsec}{\smallskip}
\AtBeginEnvironment{example}{\smallskip}
\AtBeginEnvironment{remark}{\smallskip}
\AtBeginEnvironment{algorithm}{\smallskip}
\AtBeginEnvironment{lemma}{\smallskip}
\AtBeginEnvironment{corollary}{\smallskip}

\newtheorem{theorem}{Theorem}[section]
\newtheorem{corollary}[theorem]{Corollary}
\newtheorem{lemma}[theorem]{Lemma}
\newtheorem{proposition}[theorem]{Proposition}
\theoremstyle{definition}
\newtheorem{definition}[theorem]{Definition}
\newtheorem{example}[theorem]{Example}
\newtheorem{remark}[theorem]{Remark}
\newtheorem{algorithm}[theorem]{Algorithm}

\renewcommand{\P}{\mathcal{P}}
\newcommand{\Q}{\mathcal{Q}}
\newcommand{\C}{\mathcal{C}}
\newcommand{\D}{\mathcal{D}}
\newcommand{\F}{\mathcal{F}}
\renewcommand{\S}{\mathcal{S}}
\newcommand{\dHop}{\mathrm{d_H}}
\newcommand{\conv}{\mathrm{conv}}
\renewcommand{\d}{\mathrm{d}}

\newcommand{\arc}[1]{\overset{\frown}{#1}}
\newcommand{\larc}[1]{\ell(\arc{#1})}
\newcommand{\dH}[1]{\mathrm{d_H}(#1)}

\newcommand{\titlename}	
{An algorithm for approximating the metric projection onto a superelliptic disk}
\newcommand{\shorttitlename}
{An algorithm for approximating the metric projection onto a superelliptic disk}

\newcommand{\authorname}      {Valerian-Alin Fodor$^1$ \and Virgilius-Aurelian Minuță$^2$}
\newcommand{\pdfauthorname}   {Valerian-Alin Fodor, Virgilius-Aurelian Minuta}
\newcommand{\shortauthorname} {V. A. Fodor\and V. A. Minuță}

\newcommand{\universitynameA}  {$^{1,2}$Babeș-Bolyai University}
\newcommand{\facultynameA}     {Faculty of Mathematics and Computer Science}
\newcommand{\departmentnameA}  {Department of Mathematics}
\newcommand{\addressA}  	   {M. Kogălniceanu 1, RO-400084, Cluj-Napoca, Romania}

\newcommand{\emailaddressA}    {valerian.fodor@ubbcluj.ro}
\newcommand{\emailaddressB}    {virgilius.minuta@ubbcluj.ro}

\newcommand{\articleabstract}{We study conditions under which the metric projection onto a strictly convex body can be approximated by the projections onto polygonal approximations (in the Hausdorff-Pompeiu sense) of the strictly convex body itself. We provide a necessary condition by proving that the union of all the vertices of the polygonal approximations must be dense in the boundary of that set.
To showcase our result, given a superellipse (Lamé curve) of order \texorpdfstring{$p>1$}{p>1}, we prove that a sequence of polygons obtained by radially projecting the torsion subgroup of the unit circle onto the superellipse, converges to the superellipse, in the Hausdorff-Pompeiu metric. Moreover, this leads to a globally convergent new algorithm useful in approximating the metric projection onto a superelliptic disk (the convex hull of a superellipse).}

\newcommand{\msc}{41A50, 52A27, 52A10, 52Bxx.}

\newcommand{\keywordterms}{Superellipse, Metric projection, Hausdorff-Pompeiu metric, Polygonal approximations.}

\def\depA{\departmentnameA}
\StrLen{\depA}[\depAlen]
\def\facA{\facultynameA}
\StrLen{\facA}[\facAlen]

\newcommand{\institutionA}{
\universitynameA\\
\ifthenelse{\facAlen>0}{\facultynameA\\}{}
\ifthenelse{\depAlen>0}{\departmentnameA\\}{}
\addressA}
\hypersetup{colorlinks = true, linkcolor = blue, anchorcolor =red, citecolor = blue, filecolor = red, urlcolor = blue, pdfauthor=\pdfauthorname, pdftitle=\titlename, pdfkeywords=\keywordterms, pdfsubject=\articleabstract}
\titleformat{\section}{\Large\bfseries}{\thesection}{1em}{}
\titleformat{\subsection}{\large\it}{\thesubsection}{1em}{}
\fancypagestyle{mypagestyle}{
  \fancyhf{}
  \fancyhead[OC]{\textit{\shorttitlename}}
  \fancyhead[EC]{\textit{\shortauthorname}}
  \fancyfoot[C]{\medskip $\leftslice\,$\thepage$\,\rightslice$ }
  
}
\fancypagestyle{firstpagestyle}{
  \fancyhf{}
  \fancyfoot[C]{\medskip $\leftslice\,$\thepage$\,\rightslice$ }
  
}
\title[\shorttitlename]{\LARGE{\titlename}}
\author[\shortauthorname]{\large{\authorname}
\medskip\\
{\footnotesize \institutionA\medskip\\
$\begin{array}{l}
\text{email}^1\text{: \texttt{\href{mailto:\emailaddressA}{\emailaddressA}}}\\
\text{email}^2\text{: \texttt{\href{mailto:\emailaddressB}{\emailaddressB}}}
\end{array}$
}}
\thanks{$^{1}$ Corresponding author}
\thanks{$^{2}$ This author was supported by a grant of the Romanian Ministry of Research, Innovation and Digitalization, CNCS-UEFISCDI, project number PN-IV-P1-PCE-2023-0060, within PNCDI IV}
\begin{document}
\begin{abstract}
\articleabstract\\[0.2cm]
\textsc{MSC 2010.} \msc\\[0.2cm]
\textsc{Keywords.} \keywordterms
\end{abstract}
\begingroup
\def\uppercasenonmath#1{} 
\let\MakeUppercase\relax 
\maketitle
\endgroup
\thispagestyle{firstpagestyle}

\titleformat{\section}
{\normalfont\Large\bfseries}
{\thesection.}   
{1em}
{}

\titleformat{\subsection}
{\bfseries}
{\thesubsection.}
{1em}
{}

\titlespacing*{\section}
  {0pt}      
  {6ex}      
  {2ex}      

\titlespacing*{\subsection}
  {0pt}      
  {4ex}      
  {2ex}      

\section{Introduction} \label{sec:introduction}

The topic of metric projections is wide and it has been extensively studied in various contexts. In what follows, we briefly recall several results that are relevant to our approach. 

Metric projections onto convex sets and convex bodies have been investigated from different perspectives, including the existence, uniqueness, continuity, and approximation properties of the metric projection operator. In particular, for convex bodies, various properties of the metric projection operator have been studied (see for example \cite{article:Balestro} and the references therein). In the class of convex bodies, of particular interest is also the superelliptic disk, which has a wide range of applications (e.g. mechanics  \cite{article:Mechanics}, 3D graphics  \cite{article:3D_Graphics}, particle physics \cite{article:Particle_Physics}, etc.). Denoting by $\D_p$ the superelliptic disk of order $p>1$, i.e., the convex hull of a superellipse of order $p$, denoted by $\C_p$, also known as Lamé curve; the metric projection onto $\D_p$, as stated in \cite{article:Fodor-Pintea2025}, for an arbitrary point outside of the superelliptic disk seems to be a high complexity task (e.g., even in the case of $p=2$, this is equivalent to finding an explicit solution of a degree $4$ polynomial equation, which, even though solvable by radicals, it is difficult to obtain computationally).

In this article, we propose an algorithm for approximating the metric projection onto a superelliptic disk $\D_p$. For $p\geq 1$, these sets are convex bodies.
The proposed method is based on approximating the superellipse, by a sequence of polygons that converge to it in the Hausdorff-Pompeiu metric. 
The metric projection is then computed onto these polygonal approximations, yielding approximations of the metric projection onto the superellipse. 

For the computation of metric projections onto polyhedral sets, explicit formulas can be found for example in Section~6.41 of \cite{book:Deutsch2001}, and in \cite{article:Rutkowski2017,article:YuWang2025}. 

The idea behind the existence of such polygonal approximations of the superellipse comes from the fact that in the Hausdorff-Pompeiu metric, for a sequence of polygons with vertices in the boundary of a nonempty, compact and strictly convex subset $K$ of $\mathbb{R}^2$, to converge to the boundary of $K$, i.e., $\partial K$, the set of all vertices of this polygons should be dense in $\partial K$ (see Theorem \ref{th:densitate}). Such a dense subset of points of the superellipse, can be created by radially projecting the torsion subgroup of the unit circle onto the superellipse (see Corollary \ref{cor:dense_superellipse}). Note that different dense subsets of points of the superellipse can be chosen, hence leading to different variants of the current algorithm. More generally, the algorithm can be modified to work on any nonempty, closed, bounded and convex subset of a Hilbert space, if given a good sequence of polygonal approximations (see Theorem \ref{Theorem 1.25.11.2025}).

\medskip
This paper is organized as follows. Section \ref{sec:notations} introduces the Convex Analysis framework of our algorithm, i.e. generalities about the metric projection onto a subset $C$ of a Hilbert space $X$, $P_C$; why in our paper it behaves as a single valued mapping when given additional constraints on $C$ (Remark \ref{remark:singleton}); and finally, if $X$ is finite dimensional, we prove that the inverse images through the metric projection, onto a nonempty closed convex set $C$, of the relative interiors of the nonempty faces of $C$ form a partition of $X$.

In Section \ref{sec:Hausdorff_approximation}, by using the Hausdorff-Pompeiu metric, we prove in Theorem \ref{Theorem 1.25.11.2025} that, in our specific case, the operator $C\mapsto P_C(x)$  is continuous, for every $x\in X$, i.e., if given a nonempty, closed, bounded and convex subset $C$ of a Hilbert space $X$, and a sequence of nonempty, closed, bounded and convex subsets of $X$, that converge in the Hausdorff-Pompeiu metric to $C$, then the metric projection onto this subsets will also converge to the metric projection onto $C$, thus yielding good approximations of it.

In Section \ref{sec:polygonal_convex_bodies}, we move onto the Euclidean space $\mathbb{R}^2$, where if $K\subseteq \mathbb{R}^2$ is a strictly convex body and if a sequence of convex, full-dimensional polygons with vertices on the boundary of $K$, converges in the Hausdorff-Pompeiu metric to $\partial K$, then the set of all the vertices of these polygons is dense in $\partial K$, thus providing a necessary condition for when such a sequence of polygons will provide a good approximation of $\partial K$, hence by Corollary \ref{cor:polygonal_convex_bodies}, the metric projection onto their convex hulls  will provide a good approximation of the metric projection onto $K$. Moreover, here, in Proposition \ref{prop:partition_R_minus_polyhedron}, we also prove that the inverse images of the metric projection onto the convex hull of such a polygon will form a partition of the exterior of said convex hull.

In Section \ref{sec:superellipse}, we introduce our polygonal approximations of the superellipse: We take the radial diffeomorphism $\varphi_p$ from the unit circle to the superellipse $\C_p$ (see Remark \ref{remark:diffeomorphism}), via which we radially project the torsion subgroup of the unit circle, onto $\C_p$, obtaining thus a dense subset of $\C_p$, thereby satisfying the necessary condition from Section \ref{sec:polygonal_convex_bodies}. Out of these projection points, by choosing the ones corresponding to the $k$-th roots of unity (for a natural number $k\geq 3$), we will obtain the vertices of the superellipse inscribed polygon $\P_{p,k}$. 

We continue by describing the inverse images under the metric projection onto $\Q_{p,k}$ (the convex hull of $\P_{p,k}$) corresponding to the relative interiors of the facets and to the vertices of $\P_{p,k}$ (Remark \ref{remark:region_inequalities}), which form a partition of $\mathbb{R}^2\setminus \Q_{p,k}$ (Proposition \ref{prop:partition_R_minus_polyhedron}) and we conclude by proving that the sequence of polygons $(\mathcal{P}_{p,k})_{k\ge 3}$ converges, in the Hausdorff-Pompeiu metric, to the superellipse $\C_{p}$ (Theorem \ref{convergence of the polygons}).

These results allow us in Section \ref{sec:algorithm}, to formulate our algorithm (Algorithm  \ref{algo:main}), to prove its convergence (Theorem \ref{th:algorithm_convergence}), and to give an example of utilization containing also absolute error measurements (Example \ref{example:algo_main_with_exact_errors}).

\medskip
Finally, it is to be noted that the algorithm exhibits logarithmic complexity. A detailed complexity and numerical analysis, including error estimates, comparisons to other numerical methods, speed analysis and three dimensional applications, is deferred to a forthcoming Numerical Calculus paper, while the current paper focuses on the Convex Analysis foundations of the algorithm, namely in formulating the algorithm, proving its convergence and showcasing its utilization.

\section{Notations and preliminaries} \label{sec:notations}

In this section, we consider a real Hilbert space $X$.

Let \( C \) be a nonempty subset of \( X \) and let \( x \in X \). An element \( y_0 \in C \) is called an \emph{element of best approximation} to \( x \) from \( C \) if
$
\|x - y_0\| = \inf_{y \in C} \|x - y\|.
$

The problem of best approximation of $x$ by elements of $C$ consists in finding all the elements of best approximation from \( x \) to \( C \). The solution set is denoted by 
$$
P_C(x) := \left\{ y_0 \in C \mid \|x - y_0\| = \inf_{y \in C} \|x - y\| \right\}
$$
and it is called $\emph{the metric projection}$ of $x$ on $C$.

If every \( x \in X \) admits at least one element of best approximation from \( C \), then \( C \) is called a \emph{proximinal} set. If every  \( x \in X \) admits exactly one element of best approximation from \( C \), then \( C \) is called a \emph{Chebyshev} set.

Next, we recall some well known results of Convex Analysis (see for example \cite{book:Deutsch2001}).

\begin{proposition}\label{prop:proxisclosed}
If \( C \) is a proximinal set, then \( C \) is a closed set.
\end{proposition}

\begin{proposition}\label{prop:best_one}
If \( C \) is a convex set, then every \( x \in X \) has at most one best approximation in \( C \).
\end{proposition}

\begin{proposition}
\label{prop:closed_convex_implies_Chebyshev}
If \( C \) is a closed and convex set, then $C$ is a Chebyshev set.
\end{proposition}

\begin{proposition}
\label{prop:Chebyshev_implies_closed_convex_finite_dimensional}
If $X$ is finite dimensional, then $C$ is a Chebyshev set if and only if $C$ is a closed and convex set.
\end{proposition}

\begin{remark} \label{remark:singleton}
If $C$ is a closed and convex set, then, according to Proposition \ref{prop:closed_convex_implies_Chebyshev}, for all $x \in X$, $P_C(x)$ is a singleton. Therefore, in this case, $P_C$ is a single valued mapping. 
\end{remark}

\begin{proposition}\label{prop:Charbestappconvsets}
If $C$ is a convex set, $x^* \in X$, and $x^0 \in C$, then $ P_C(x^*) = x^0$ if and only if
$$
\langle  x - x^0, x^* - x^0 \rangle \leq 0\text{, for all  }x \in C.
$$
\end{proposition}

\begin{remark}\label{remarknormalcone}
From a geometric point of view, the property of Proposition \ref{prop:Charbestappconvsets}, that  
$
\langle  x - x^0, x^* - x^0 \rangle \leq 0$, for all  $x \in C.
$ 
shows that  
$x^*-x^0$ belongs to the normal cone to $C$ at $x^0$, i.e., 
$$N_C(x^0) = \{d \in X \mid  \langle x-x^0,d \rangle \leq 0 \text{ for all } x\in
C\},$$
which is a closed convex cone (see, e.g., \cite{book:BarbuPrecupanu2012}).
\end{remark}

For a subset $S$ of $X$, the set defined by $$\mathrm{rint}(S) = \{ x \in S \mid \exists \varepsilon > 0 \, \text{such that} \, B_{\varepsilon}(x) \cap \mathrm{aff}(S) \subseteq S \}$$
is called the \textit{relative interior} of $S$. Here, $B_{\varepsilon}(x)$ denotes the open ball centered at $x$ and of radius $\varepsilon$, and $\mathrm{aff}(S)$ denotes the affine hull of $S$. Recall that if $X$ is finite dimensional and $S$ is a nonempty convex set, then $\mathrm{rint}(S)$ is nonempty (see \cite[Theorem 6.2]{book:Rockafellar70}).

\begin{definition}
    We say that $C$ is a \textit{convex body}, if it is convex, compact and with a nonempty interior. Similarly, we say that $C$ is a \textit{strictly convex body}, if it is strictly convex, compact and with a nonempty interior.
\end{definition}

\begin{definition}
If $C$ is a convex set, then a convex subset $ F \subseteq C $ is a \textit{face} of $ C $ if for every $t \in F$ and for every $ x, y \in C $ such that $t$ is in the open segment $(x,y)$ we have that $[x,y] \subseteq F$.
\end{definition}

Finally, if $X$ is finite dimensional, in what follows we prove that the inverse images through the metric projection, onto a nonempty closed convex set $C$, of the relative interiors of the nonempty faces of $C$ form a partition of $X$. This will be of use, later in the article.

\begin{proposition}\label{prop:partition_R_n}
If $X$ is finite dimensional and $C \subseteq X$ is a nonempty closed and convex set, then, the set $$M = \{P_C^{-1} (\mathrm{rint}(F)) \mid F \,\text{ is a nonempty face of } \, C \}
$$
forms a partition of $X$.
\end{proposition}
\begin{proof}
Since the relative interiors of nonempty faces, of the closed convex set $C$ form a partition of $C$ (\cite[Theorem 2.6.10]{book:Webster1994}) and $C$ is nonempty and closed, then, for all $x \in X$, there exists a nonempty face $F$ of $C$ such that $P_C(x) \in \mathrm{rint}(F)$, i.e., $x \in P_C^{-1} (\mathrm{rint}(F)).$ Thus, 
$$\bigcup_{F \text{ is a face of }C} P_C^{-1} (\mathrm{rint}(F)) = X.$$ Furthermore, since $C$ is also convex, the sets in $M$ are disjoint. 
Also, for every nonempty face $F$ of $C$, if $x \in \mathrm{rint}(F)$, then $P_C(x) = x$ which implies that $x \in P_C^{-1} (\mathrm{rint}(F)).$ So, $P_C^{-1} (\mathrm{rint}(F))$ is nonempty.
\end{proof}

\section{Approximating the metric projection\texorpdfstring{\\}{ }via the Hausdorff-Pompeiu distance}
\label{sec:Hausdorff_approximation}

As in the previous section, we consider a real Hilbert space $X$. In what follows, we recall the Hausdorff-Pompeiu distance and we prove that in certain conditions for an arbitrary point $x\in X$, the operator $C\mapsto P_C(x)$ is continuous

\begin{definition}
Let $A$ and $B$ be two nonempty subsets of $X$. The Hausdorff-Pompeiu distance between $A$ and $B$ is denoted by $\dHop$ and is defined by
$$
\dH{A,B}
=
\max\left\{
\sup_{x\in A}\inf_{y\in B}\|x-y\|,
\;
\sup_{y\in B}\inf_{x\in A}\|x-y\|
\right\}.
$$
\end{definition}

\begin{lemma}\label{lemma:ball}
If $A$ and $B$ are two nonempty closed sets, $\varepsilon$ is a positive real number  and $\dH{A,B} < \varepsilon$, then
$$
A \subseteq B + B_{\varepsilon}. 
$$
where $B_{\varepsilon}$ is the open ball centered in $0$ of radius $\varepsilon.$
\end{lemma}

\begin{proof}
Indeed, if $x \in A,$ let $y_0 = P_B(x).$ Then, by definition, 
$$\|x-y_0\| = \inf_{y\in B}\|x-y\| \le \dH{A,B} < \varepsilon.$$
Thus, $x = y_0 + x - y_0 \in B + B_{\varepsilon}$.
\end{proof}

\begin{theorem}\label{Theorem 1.25.11.2025}
Let $(C_n)_{n\ge 1}$ be a sequence of nonempty, closed, bounded and convex 
subsets of $X$, and let $C$ be a nonempty, closed and convex subset of $X$ 
such that
$$
\dH{C_n, C}\longrightarrow 0 \text{ as } n\to\infty.
$$
Then $C$ is bounded, and for every $x\in X$, the sequence of metric projections converges:
$$
P_{C_n}(x)\longrightarrow P_C(x) \text{ as } n\to\infty.
$$
\end{theorem}
\begin{proof}
For all $n\geq 1$, let $y_n = P_{C_n}(0)$ and $y = P_C(0)$. We will prove that $y_n \longrightarrow y.$
Fix $\varepsilon>0$. Since $C_n \to C$, there exists $n_0\in\mathbb{N}$ such that
$\dH{C_n,C} < \varepsilon$, for all  $n \ge n_0$ .

Let $x_n = P_{C_n}(y)$ be the metric projection of $y$ on $C_n$.  
By the definitions of the metric projection and of the Hausdorff distance, for all $n \geq n_0$, we have
$
  \|x_n - y\| \le \dH{C_n,C} < \varepsilon .
$
Therefore
$
  \|y_n\| \le \|x_n\|
  = \|y + x_n - y\|
  \le \|y\| + \|x_n - y\|
  < \|y\| + \varepsilon ,
$
so
\begin{equation}\label{eq:1}
  \|y_n\| - \|y\| < \varepsilon .
\end{equation}

Now let $z_n = P_C(y_n)$ be the metric projection of $y_n$ on $C$.
By the definitions of the metric projection and of the Hausdorff distance, for all $n \geq n_0$, we have
$
  \|y_n - z_n\| \le \dH{C_n,C} < \varepsilon .
$
It follows that
$
  \|y\| \le \|z_n\|
  = \|y_n + z_n - y_n\|
  \le \|y_n\| + \|z_n - y_n\|
  < \|y_n\| + \varepsilon ,
$
hence
\begin{equation}\label{eq:2}
  \|y\| - \|y_n\| < \varepsilon .
\end{equation}

From \eqref{eq:1} and \eqref{eq:2} we obtain
$
  \left|\|y_n\| - \|y\|\right| < \varepsilon$, for all $n \ge n_0,
$
so $\|y_n\|\to\|y\|$, as  $n\to\infty$.

Let $(y_{n_k})$ be an arbitrary weakly convergent subsequence of $(y_n)$,
and assume that $y_{n_k}\rightharpoonup z$.
For each $k$, define $z_{n_k} := P_C(y_{n_k})$.
By the definition of the Hausdorff distance, we have
\[
\|y_{n_k}-z_{n_k}\| \le \dH{C_{n_k},C} \xrightarrow{k\to\infty} 0.
\]
Hence $z_{n_k} - y_{n_k} \to 0$, and since strong convergence implies weak convergence, $z_{n_k} - y_{n_k} \rightharpoonup 0.$ Thus, $z_{n_k} \rightharpoonup z$, since $y_{n_k} \rightharpoonup z.$
Since $C$ is closed and convex, it is weakly closed, and therefore $z \in C$. 

By the weak lower semicontinuity of the norm, we obtain
$
\|z\| \le \liminf_{k\to\infty} \|y_{n_k}\|
       = \|y\|.
$
On the other hand, by definition 
$
\|y\| \le \|w\|, 
$ for all  $w \in C$,
and in particular $\|y\| \le \|z\|$.
Consequently, $\|z\| = \|y\|$.
Since the metric projection onto a closed and convex set is unique (Proposition \ref{prop:closed_convex_implies_Chebyshev}), it follows that
$z = y$.
Thus $y_{n_k} \rightharpoonup y$.

So, we have proved that every weakly convergent subsequence of $(y_n)$ converges weakly to $y$. 

Note also, that $(y_n)$ is bounded, since for all $n \geq n_0$ we have that $C_n \subseteq C + B_{\varepsilon}$ (see Lemma \ref{lemma:ball}).

Thus, we conclude that $y_n \rightharpoonup y$. Otherwise, if $(y_n)$ does not converge weakly to $y$, then, there exists a subsequence $(y_{n_l})$ such that it does not converge weakly to $y$. Then, there exists a continuous linear functional $f$ such that $(f(y_{n_l}))$ does not converge to $f (y)$. By passing through a subsequence, we can assume that 
\begin{equation}\label{eq:3}
|f(y_{n_l}) - f(y)| \geq \varepsilon\quad \text{ for all }l\in\mathbb{N}.
\end{equation}

But, since $(y_n)$ is bounded, so is $(y_{n_l})$. Thus, it has a weakly convergent subsequence. We have proved that every weakly convergent subsequence of $(y_n)$ converges weakly to $y$, hence even this subsequence converges weakly to $y$, which contradicts \eqref{eq:3}.

Furthermore, $y_n \rightharpoonup y$ together with the convergence $\|y_n\| \to \|y\|$, implies that
$y_n \to y$ strongly.

Now consider the sets $D_n := C_n - x$ for all $n \ge 1$, and $D := C - x$.
The sequence $(D_n)$ consists of nonempty, closed, bounded, and convex sets
and converges to the nonempty, closed, bounded, and convex set $D$
with respect to the Hausdorff metric. 
Applying the previous result, we have
\[
P_{D_n}(0) \longrightarrow P_D(0)  \text{ as } n\to\infty,
\]
which means that
\[
P_{C_n}(x) \longrightarrow P_C(x)  \text{ as } n\to\infty.\qedhere
\]
\end{proof}

\section{On the polygonal approximations of strictly convex bodies}
\label{sec:polygonal_convex_bodies}

In this section, we take $X=\mathbb{R}^2$, endowed with the usual Euclidean norm. 

In this framework, we consider any polygon to be convex (in the geometrical sense) and full-dimensional, i.e., the dimension of the affine hull of the polygon is $2$. Moreover, we will call the convex hull of a polygon a \textit{polyhedron}, and, in notations, if $\P$ is a polygon, then $\Q=\conv(\P)$ will be its associated polyhedron. In this case, $\Q$ is nonempty, closed and convex, hence, by Proposition \ref{prop:closed_convex_implies_Chebyshev}, $\Q$ is a Chebyshev set and therefore, by Remark \ref{remark:singleton}, the metric projection onto $\Q$ is a single valued mapping.

We begin by proving that the inverse images through the metric projection, onto a polyhedron $Q$, of the relative interiors of the vertices and of the edges of $Q$ form a partition of $\mathbb{R}^2\setminus\Q$. This will be of use, later in the article.

\goodbreak
\begin{proposition} \label{prop:partition_R_minus_polyhedron}
Let $\Q$ be a polyhedron, and let $P_{\Q}$ denote the metric 
projection onto $\Q$.

For each vertex $v$ of $\Q$, define
$$
\mathcal{V}(v)
=
\left\{
x \in \mathbb{R}^2 \setminus \mathcal{Q}
\,\middle|\,
P_{\mathcal{Q}}(x) = v
\right\},
$$
and for each edge $E$ of $\mathcal{Q}$, define
$$
\mathcal{E}(E)
=
\left\{
x \in \mathbb{R}^2 \setminus \mathcal{Q}
\,\middle|\,
P_{\mathcal{Q}}(x) \in \mathrm{rint}(E)
\right\}.
$$

Then the family of sets
$$
\{\mathcal{V}(v) \mid v \text{ vertex of } \mathcal{Q}\}
\;\cup\;
\{\mathcal{E}(E) \mid E \text{ edge of } \mathcal{Q}\}
$$
forms a partition of $\mathbb{R}^2 \setminus \mathcal{Q}$.
\end{proposition}
\begin{proof}
The nonempty faces of $\mathcal{Q}$ are $\mathcal{Q}$, the edges of $\mathcal{Q}$, and the vertices of $\mathcal{Q}$. Note that since $\mathcal{Q}$ is full-dimensional, then $\mathrm{rint} (\mathcal{Q})=\mathrm{int} (\mathcal{Q})$. Since $x \in P_{\mathcal{Q}}^{-1}(\mathrm{rint}(\mathcal{Q}))$ if and only if $x \in \mathrm{rint} (\mathcal{Q}),$ if we consider the sets $P_{\mathcal{Q}}^{-1}(\mathrm{rint}(F)) \setminus \mathcal{Q}$, for all faces $F$ of $\mathcal{Q},$ by Lemma \ref{prop:partition_R_n} the conclusion follows immediately. 
\end{proof}

We continue, by providing an immediate consequence of Theorem \ref{Theorem 1.25.11.2025}.

\begin{corollary}\label{cor:polygonal_convex_bodies}
Consider $K\subseteq \mathbb{R}^2$ a convex body.
Let $(\mathcal{Q}_n)_{n\ge 1}$ be a sequence of polyhedrons with vertices on the boundary of $K$, i.e., $\partial K$, converging to $K$ in the
Hausdorff-Pompeiu metric. Then for every $x\in X$, the sequence of metric projections converges:
$$
P_{\Q_n}(x)\longrightarrow P_K(x) \text{ as } n\to\infty.
$$
\end{corollary}

In what follows, we prove that if a subset $K\subseteq \mathbb{R}^2$ is a strictly convex body, and if its boundary is approximated in the Hausdorff-Pompeiu sense by a sequence of polygons whose vertices lie on the boundary, then the union of all the vertices of these polygons is dense in the boundary (Theorem \ref{th:densitate}).

\begin{lemma}\label{lema:distdelamidpointlafrontieră}
Consider $K\subseteq \mathbb{R}^2$ a nonempty compact and strictly convex set. Let $d > 0.$ Then, there exists $m>0$ such that 
$$\d\left(\frac{x+y}{2}, \partial K\right) \ge m,$$ 
for every $x,y \in \partial K,$ satisfying $\|x-y\| \ge d.$     
\end{lemma}
\begin{proof}
 Consider the functions $$
T:\mathbb{R}^2\times \mathbb{R}^2 \longrightarrow \mathbb{R}^2,
\qquad
T(x,y)=\frac{x+y}{2}.
$$
and 
$$
\varphi:\mathbb{R}^2\longrightarrow \mathbb{R},
\qquad
\varphi(x)=\d(x,\partial K),
$$
and the set 
$$
M_d=\Bigl\{(x,y)\in \partial K\times \partial K
\;\Big|\;
\|x-y\|\ge d
\Bigr\}.
$$

Observe that the set $M_d$ is compact, since it is a closed subset of a compact set. 
Also, the map \(T\) is continuous, being linear on a finite-dimensional space, and the distance function $\varphi$ is continuous. Thus, the function \[
\psi:M_d\longrightarrow \mathbb{R},
\qquad
\psi(x,y)
=
\frac{(\varphi\circ T)(x,y)}
{\|x-y\|},
\]
is well-defined and continuous. 
The compactness of $M_d$, implies the existence of $c = (a,b)\in M_d$ such that
\[
\psi(c)=\min_{u\in M_d}\psi(u).
\]

Since $K$ is a strictly convex set, $\d\left(\frac{a+b}{2}, \partial K\right) > 0,$ which implies that $\psi(c) >0.$
Consequently,
\[
\d\left(\frac{x+y}{2},\partial K\right)\ge \psi(c)d >0,
\]
 for every $x,y\in \partial K$ with the property $\|x-y\| \ge d.$ 
\end{proof}
\begin{lemma}\label{lema:maxdistconvto0}
Consider $K\subseteq \mathbb{R}^2$ a strictly convex body.
Let $(\mathcal{P}_n)_{n\ge 1}$ be a sequence of polygons with vertices on the boundary of $K$, converging to $\partial K$ in the
Hausdorff-Pompeiu metric. Then the length of the longest edge of $\mathcal{P}_n$
converges to zero as $n\to\infty$.
\end{lemma}
\begin{proof}
Assume, by contradiction, that the conclusion does not hold. Then there exists a subsequence $(\mathcal{P}_{n_k})_{k\geq 1}$ and a constant $d>0$ such that, for every $k\geq 1$, one edge $[P_iP_{i+1}]$ of the polygon $\mathcal{P}_{n_k}$ satisfies
\[
\|P_i-P_{i+1}\|\geq d.
\]

By Lemma \ref{lema:distdelamidpointlafrontieră}, there exists $m>0$ such that, for every $k \ge 1$
\[
\d(M,\partial K)\geq m,
\]
where $M$ is the midpoint of the edge $[P_iP_{i+1}]$. Thus, for every $k \ge 1,$ $$\dH{\mathcal{P}_{n_k},\, \partial K} \ge m >0.$$

This contradicts the fact that $(P_n)_{n\ge 1}$ converges to $\partial K$ in the Hausdorff-Pompeiu metric.
\end{proof}

\begin{theorem} \label{th:densitate}
Consider $K\subseteq \mathbb{R}^2$ a strictly convex body.
Let $(\mathcal{P}_n)_{n\ge 1}$ be a sequence of polygons with vertices on $\partial K$, converging to $\partial K$ in the
Hausdorff-Pompeiu metric. 
Then, the set of all the vertices of the polygons $(\mathcal{P}_n)_{n\ge 1}$ is dense in $\partial K$.
\end{theorem}
\begin{proof}
Let $P\in \partial K$ and let $\ell\in\mathbb N^{*}$.
Since $(\mathcal{P}_n)_{n \ge 1}$ converges to $\partial K$ in the Hausdorff-Pompeiu metric, there exists
$n_1\in\mathbb N^{*}$ such that for every $n\geq n_1$,
there exists a point $P_n=P_{\mathcal{P}_n}(P)\in \mathcal{P}_n$ satisfying
\[
\d(P,\mathcal{P}_n)=\|P -P_{\mathcal{P}_n}(P)\|=\|P- P_n\|<\frac1{2\ell}.
\]

By Lemma \ref{lema:maxdistconvto0}, there exists $n_2\in\mathbb N^{*}$ such that for every
$n\geq n_2$, every edge of the polygon $P_n$ has length less than $\frac{1}{2\ell}.$

Fix $n_0 =\max\{n_1,n_2\}$ and let $n\geq n_0$ be an arbitrary natural number. Let $P_i$ and $P_{i+1}$ be two consecutive
vertices of $\mathcal{P}_n$ such that
$$
P_{\mathcal{P}_n}(P)\in [P_i,P_{i+1}].
$$

We know that
$$
\|P_i-P\|
\leq
\|P_i-P_{\mathcal{P}_n}(P)\|
+
\|P_{\mathcal{P}_n}(P)-P\|.
$$
Since $P_{\mathcal{P}_n}(P)\in [P_i,P_{i+1}]$,  by the choice of $n_0$, we have
$$
\|P_i-P_{\mathcal{P}_n}(P)\|
\leq
\|P_i-P_{i+1}\|
<
\frac1{2\ell},
$$
and
$$
\|P_{\mathcal{P}_n}(P)-P\|
<
\frac1{2\ell}.
$$
Therefore,
$$
\|P_i-P\|
<
\frac1\ell.
$$

We have thus shown that for every $P\in\partial K$ and every
$\ell \in \mathbb{N}^{*}$, there exists $n_0\in\mathbb{N}$ such that for every $n\geq n_0$, there exists a vertex $P_i$ of polygon
$\mathcal{P}_n$ such that
$$
\|P_i-P\|<\frac1\ell.
$$
Hence, for every $P\in\partial K$ and every
$\ell \in \mathbb{N}^{*}$, there exists $P_i\in \bigcup_{n\geq 1} \mathcal{P}_n$ such that
$$
\|P_i-P\|<\frac1\ell.
$$

In conclusion, the set of all the vertices of the polygons $(\mathcal{P}_n)_{n\ge 1}$ is dense in $\partial K$.
\end{proof}

\section{Polygonal approximations of superellipses} \label{sec:superellipse}

    Throughout this section and the sections that follow, if not stated otherwise, our framework is the $\mathbb{R}^2$ Euclidean space endowed with the usual inner product. Moreover, on $\mathbb R^2$ we consider the weighted $p$-norm ($p\in\mathbb{R}$, $p\geq 1$)
$$
\|(x,y)\|_{p}
=
\sqrt[p]{\left|\frac{x^{p}}{a^{p}}\right|+\left|\frac{y^{p}}{b^{p}}\right|}.
$$

Let $a$ and $b$ be two strictly positive real numbers, and $p > 1$ a real number. We denote by
$$
\C_{p}=\left\{(x,y)\in\mathbb R^2\mid\ 
\left|\frac{x^{p}}{a^{p}}\right|+\left|\frac{y^{p}}{b^{p}}\right|=1\right\}
$$
the superellipse of order $p$, also known as the Lam\'e curve. Its convex hull, the superelliptic disk of order $p$, will be denoted by 
$$
\D_{p}=\conv(\C_{p})=\left\{(x,y)\in\mathbb R^2\mid\ 
\left|\frac{x^{p}}{a^{p}}\right|+\left|\frac{y^{p}}{b^{p}}\right|\leq 1\right\}.
$$

For two points $A$ and $B$ on the superellipse, we will denote the length of the small arc $\arc{AB}$ with $\larc{AB}$.

The unit circle is defined by $\S^1=\{(x,y)\in\mathbb R^2\mid\ x^2+y^2=1\}.$

\begin{remark}\label{remark:diffeomorphism}
The smooth function
$$
\varphi_{p}:\S^1\to \C_{p},\qquad
\varphi_{p}(u,v)=\frac{(u,v)}{\|(u,v)\|_{p}}
$$
is a diffeomorphism, which we call the \textit{radial diffeomorphism} between $\S^1$ and the superellipse $\C_{p}$, whose inverse is 
\[
\psi:\C_{p}\to \S^1,\qquad \psi(x,y)=\frac{(x,y)}{\sqrt{x^2+y^2}}.
\]
\end{remark}

Let $\mathbb{T}$ be the unit circle group corresponding to the set $\S^1$, and we denote by $$\mathrm{Tor}(\mathbb{T})=\{z\in \mathbb{C}\mid\text{ there exists a natural number }n\text{ such that }z^n=1\}$$ the torsion subgroup of $\mathbb{T}$.
Consider a natural number $k\geq 3$. We denote by $\xi$ the primitive $k$-th root of unity
$$
\xi = \cos \frac{2\pi }{k} + i \sin \frac{2\pi }{k}.
$$

We consider the subgroup of $\mathrm{Tor}(\mathbb{T})$ $$\mathbb{T}_k=\{\xi^t\mid t=0,1,\ldots,k-1\},$$
i.e. the $k$-th roots of unity. The images of these roots under the radial diffeomorphism $\varphi_{p}$ (see Remark \ref{remark:diffeomorphism}):
$$
\varphi_{p}(\xi^t) = 
\frac{1}{\|\xi^{t}\|_{p}}
\left(\cos \tfrac{2t\pi}{k},\, \sin \tfrac{2t\pi}{k}\right).
$$
generate the following points
$$
P_t  
\left(
\frac{\cos \tfrac{2t\pi}{k}}{\|\xi^{t}\|_{p}},
\frac{\sin \tfrac{2t\pi}{k}}{\|\xi^{t}\|_{p}}
\right),
\qquad
t = 0,1,\ldots,k-1.
$$

It is known that $\mathrm{Tor}(\mathbb{T})=\bigcup_{k\geq 3}\mathbb{T}_k$ is dense in $\mathbb{T}$ and that via a homeomorphism density is preserved, hence $\varphi_p(\mathrm{Tor}(\mathbb{T}))=\bigcup_{k\geq 3}\varphi_p(\mathbb{T}_k)$ is dense in $\C_p$. Therefore, by creating a sequence of polygons with vertices in $\varphi_p(\mathrm{Tor}(\mathbb{T}))$ the necessary condition given by Theorem \ref{th:densitate}, for them to converge in the Hausdorff-Pompeiu metric to $\C_p$, is verified. This sequence, will be constructed as follows. Consider the sequence $(\P_{p,k})_{k\geq 3}$, where $\P_{p,k}$ is the polygon given by $P_0P_1\ldots P_{k-1}$.

\begin{example}\label{ex:radial_projections}
Considering the superellipse 
\[
\C_4:=\left\{(x,y)\in\mathbb R^2\mid\ 
\frac{x^{4}}{625}+\frac{y^{4}}{81}=1\right\},
\]
for $k=3$ we have three points (see Figure \ref{fig:polygon_3})
\[P_0\left(5, 0\right),\quad P_1\left(-\frac{5}{634}634^{3/4}\sqrt{3}, \frac{15}{634}634^{3/4}\right),\quad P_2\left(-\frac{5}{634}634^{3/4}\sqrt{3}, -\frac{15}{634}634^{3/4}\right);\]
and for $k=4$ we obtain four points (see Figure \ref{fig:polygon_4})
\[P_0\left(5,0\right),\qquad P_1\left(0,3\right),\qquad P_2\left(-5,0\right),\qquad P_3\left(0,-3\right).\]

Note that in Figure \ref{fig:polygon} the blue points are the unity roots, in cyan we have the projection segments of the unity roots to the superellipse, the black points are the images of the $k$-th roots of unity, and, with black lines, we have the edges of the polygon.

\begin{figure}[H]
    \centering
    \begin{subfigure}{0.4\textwidth}
        \centering
        \includegraphics[width=\linewidth, trim=0 2.8cm 0 2.8cm,
  clip]{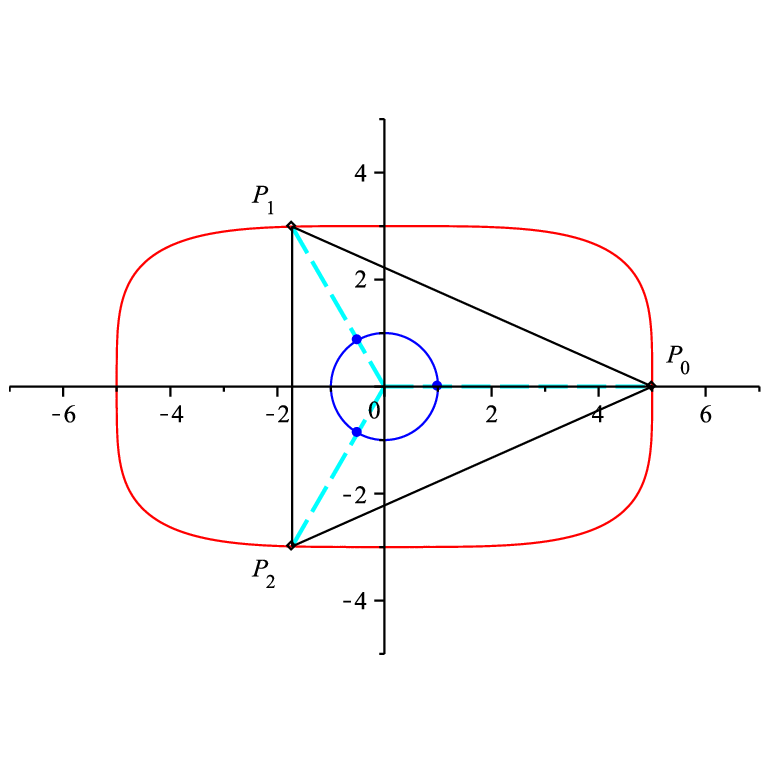}
        \caption{$k=3$}
        \label{fig:polygon_3}
    \end{subfigure}
    \qquad
    \begin{subfigure}{0.4\textwidth}
        \centering
        \includegraphics[width=\linewidth, trim=0 2.8cm 0 2.8cm,
  clip]{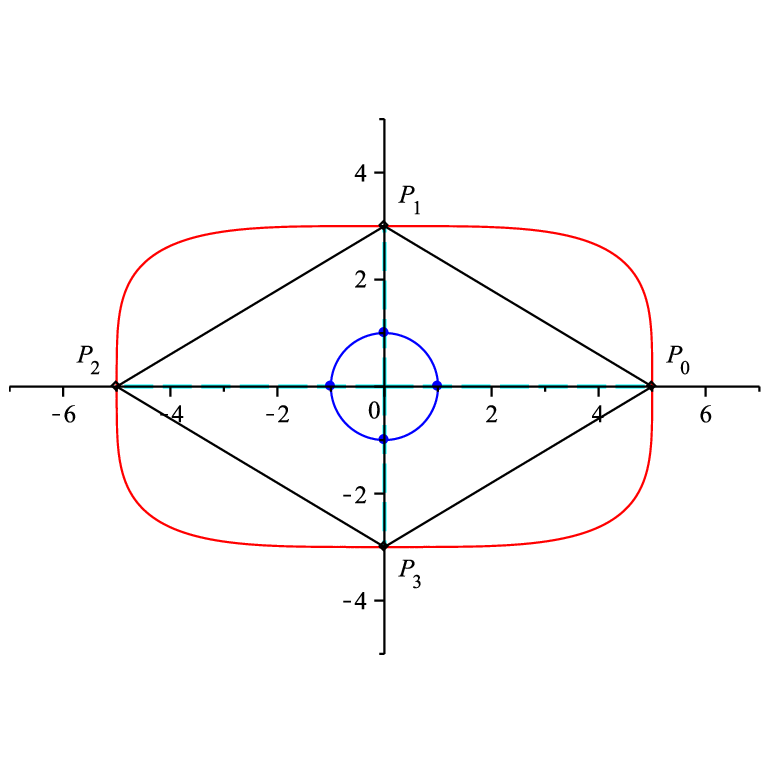}
        \caption{$k=4$}
        \label{fig:polygon_4}
    \end{subfigure}
    \caption{The polygon $\mathcal{P}_{4,k}$}
    \label{fig:polygon}
\end{figure}
\end{example}

Each edge of the polygon is denoted by
$$
[P_t,\, P_{t+1}],
\qquad t = 0,1,\ldots,k-1,
$$
with indices understood modulo $k$. Each edge determines a support line of equation:
$$
\langle a_t, x \rangle - b_t = 0,
\qquad
x = \begin{pmatrix}x\\y\end{pmatrix},
$$
where
\begin{equation}\label{eq:atbt}
a_t =
\begin{pmatrix}
\dfrac{\sin \tfrac{2(t+1)\pi}{k}}{\|\xi^{t+1}\|_{p}} -
\dfrac{\sin \tfrac{2t\pi}{k}}{\|\xi^{t}\|_{p}} \\[12pt]
-\left(
\dfrac{\cos \tfrac{2(t+1)\pi}{k}}{\|\xi^{t+1}\|_{p}} -
\dfrac{\cos \tfrac{2t\pi}{k}}{\|\xi^{t}\|_{p}}
\right)
\end{pmatrix},
\quad\text{and}\quad
b_t = \,\frac{\sin \tfrac{2\pi}{k}}{\|\xi^{t}\|_{p}\,\|\xi^{t+1}\|_{p}}.
\end{equation}

Since $\sin\!\tfrac{2\pi}{k} > 0$, we have $b_t > 0$.  
Evaluating at the origin $(0,0)$ gives $\langle a_t, 0 \rangle - b_t = - b_t < 0$, 
therefore the polyhedron lies entirely within each half-space
$$
\langle a_t, x \rangle \le b_t,
\quad \text{i.e., } 
\langle a_t, x \rangle \le 
\frac{\sin \tfrac{2\pi}{k}}{\|\xi^{t}\|_{p}\,\|\xi^{t+1}\|_{p}},
\quad \text{for all } t = 0,1,\ldots,k-1.
$$

Consequently, the \textbf{associated polyhedron} (the convex hull of the images $\varphi_{p}(\xi^t)$)
can be written as
$$
\mathcal{Q}_{p,k}=\conv(\mathcal{P}_{p,k})
=
\bigcap_{t=0}^{k-1}
\left\{
x \in \mathbb{R}^2 \;\mid \;
\langle a_t, x \rangle \le b_t
\right\}.
$$

\begin{remark}\label{remark:region_inequalities}
For each $t=0,1,\dots, k-1$, we denote by $\F_t$ the facet of the 
polyhedron $\mathcal{Q}_{p,k}$, that is, the edge $[P_t,\, P_{t+1}]$ of the polygon.
The inverse images under the metric projection onto $\mathcal{Q}_{p,k}$, corresponding to the relative interior of the facet $\F_t$, of the polyhedron, are rectangular semistrips, which we will call for simplicity \textbf{rectangular regions} (see for example Figure \ref{fig:rectangular_regions}), and are given by the following system of inequalities.

\[\label{eq:rectangular_region_inequalities}
\left\{\begin{aligned}
&\left(\frac{\sin\frac{2(t+1)\pi}{k}}{\|\xi^{t+1}\|_{p}}-\frac{\sin\frac{2t\pi}{k}}{\|\xi^{t}\|_{p}}\right)x
-\left(\frac{\cos\frac{2(t+1)\pi}{k}}{\|\xi^{t+1}\|_{p}}-\frac{\cos\frac{2t\pi}{k}}{\|\xi^{t}\|_{p}}\right)y
-\frac{\sin\frac{2\pi}{k}}{\|\xi^{t}\|_{p}\,\|\xi^{t+1}\|_{p}}
\ge 0,\\[8pt]
&\left(\frac{\cos\frac{2t\pi}{k}}{\|\xi^{t}\|_{p}} - \frac{\cos\frac{2(t+1)\pi}{k}}{\|\xi^{t+1}\|_{p}}\right)x
-\left(\frac{\sin\frac{2(t+1)\pi}{k}}{\|\xi^{t+1}\|_{p}}-\frac{\sin\frac{2t\pi}{k}}{\|\xi^{t}\|_{p}}\right)y
+\frac{\cos\frac{2\pi}{k}}{\|\xi^{t}\|_{p}\,\|\xi^{t+1}\|_{p}}
-\frac{1}{\|\xi^{t}\|_{p}^{2}}
< 0,\\[8pt]
&\left(\frac{\cos\frac{2t\pi}{k}}{\|\xi^{t}\|_{p}} - \frac{\cos\frac{2(t+1)\pi}{k}}{\|\xi^{t+1}\|_{p}}\right)x
-\left(\frac{\sin\frac{2(t+1)\pi}{k}}{\|\xi^{t+1}\|_{p}}-\frac{\sin\frac{2t\pi}{k}}{\|\xi^{t}\|_{p}}\right)y
-\frac{\cos\frac{2\pi}{k}}{\|\xi^{t}\|_{p}\,\|\xi^{t+1}\|_{p}}+\frac{1}{\|\xi^{t+1}\|_{p}^{2}} > 0.
\end{aligned}\right.
\]

The metric projection preimages of the vertices of the polyhedron are affine normal cones, which we will call for simplicity \textbf{triangular regions}  (see for example Figure \ref{fig:triangular_regions}). The affine normal cone to $\mathcal{Q}_{p,k}$ at $\varphi_{p}(\xi^t)$ is given by the following system of inequalities.
\begin{equation*}
\label{eq:triangular_region_inequalities}
\left\{\begin{aligned}
&\left(\frac{\cos\frac{2t\pi}{k}}{\|\xi^{t}\|_{p}} - \frac{\cos\frac{2(t+1)\pi}{k}}{\|\xi^{t+1}\|_{p}}\right)x
-\left(\frac{\sin\frac{2(t+1)\pi}{k}}{\|\xi^{t+1}\|_{p}}-\frac{\sin\frac{2t\pi}{k}}{\|\xi^{t}\|_{p}}\right)y
+\frac{\cos\frac{2\pi}{k}}{\|\xi^{t}\|_{p}\,\|\xi^{t+1}\|_{p}}
-\frac{1}{\|\xi^{t}\|_{p}^{2}}
\geq 0,\\[8pt]
&\left(\frac{\cos\frac{2(t-1)\pi}{k}}{\|\xi^{t-1}\|_{p}} - \frac{\cos\frac{2t\pi}{k}}{\|\xi^{t}\|_{p}}\right)x
-\left(\frac{\sin\frac{2t\pi}{k}}{\|\xi^{t}\|_{p}}-\frac{\sin\frac{2(t-1)\pi}{k}}{\|\xi^{t-1}\|_{p}}\right)y
+\frac{\cos\frac{2\pi}{k}}{\|\xi^{t}\|_{p}\,\|\xi^{t-1}\|_{p}}-\frac{1}{\|\xi^{t}\|_{p}^{2}} \leq 0.
\end{aligned}\right.
\end{equation*}

These systems of inequalities can be obtained by elementary means, through direct computations. Alternatively, they also follow from a previous work (\cite[Theorem 7.2]{article:Fodor-Pintea2025}), where the inverse image, under the metric projection, of a facet of a polyhedral set was characterized. 

\begin{figure}[H]
    \centering
    \begin{subfigure}{0.4\textwidth}
        \centering
        \includegraphics[width=\linewidth, trim=0 2.8cm 0 2.8cm,
  clip]{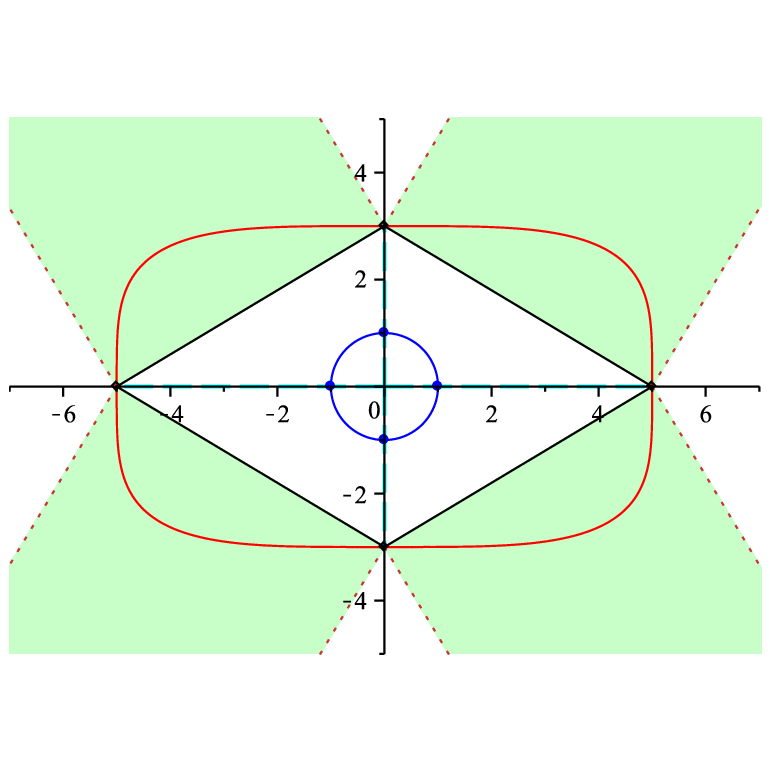}
        \caption{Rectangular regions corresponding to the facets of the polyhedron}
        \label{fig:rectangular_regions}
    \end{subfigure}
    \qquad
    \begin{subfigure}{0.4\textwidth}
        \centering
        \includegraphics[width=\linewidth, trim=0 2.8cm 0 2.8cm,
  clip]{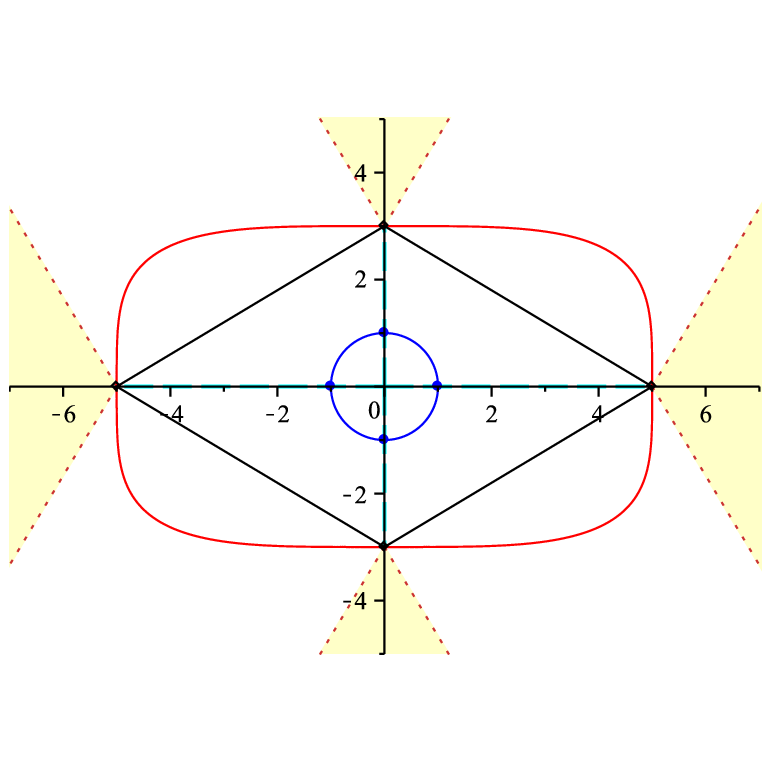}
        \caption{Triangular regions corresponding to the vertices of the polyhedron}
        \label{fig:triangular_regions}
    \end{subfigure}
    \caption{The inverse images under the metric projection onto $\mathcal{Q}_{4,4}$, for the superellipse of parameters $a=5$ and $b=3$}
    \label{fig:polyhedron_regions}
\end{figure}
\end{remark}

\begin{remark}
Given the description of the inverse images of the relative interiors of the facets and of the vertices of the polyhedron $\mathcal{Q}_{p,k}$, we may determine, for any point 
$x^*\in\mathbb{R}^2\setminus \mathcal{Q}_{p,k}$, the unique region of the 
partition to which $x^*$ belongs. 

If the coordinates of $x^*$ satisfy the inequalities of a triangular region (see Remark \ref{remark:region_inequalities}) then, $x^*$ lies in the inverse image of a vertex $P_t$, i.e., 
$P_{\D_{p}}(x^*)=P_t.$

\goodbreak
Otherwise, according to Proposition \ref{prop:partition_R_minus_polyhedron}, $x^*$ satisfies the inequalities of a rectangular region (see Remark \ref{remark:region_inequalities}), then $x^*$ lies in the inverse image of the relative interior of the facet $\mathcal{F}_t$. In this case, we compute the metric projection onto the polyhedron $\mathcal{Q}_{p,k}$. The metric projection onto the polyhedron, coincides with the metric projection onto the facet $\mathcal{F}_t$ and can be computed (see for example \cite[6.41]{book:Deutsch2001}) with the formula $$P_{\mathcal{Q}_{p,k}}(x^*) = x^* - \frac{\langle a_t, x^* \rangle - b_t}{\|a_t\|^2} a_t,
$$ 
where $a_t$ and $b_t$ are defined as in \eqref{eq:atbt}.
\end{remark}

\begin{theorem}\label{convergence of the polygons}
The sequence of polygons $(\mathcal{P}_{p,k})_{k\ge 3}$ converges, in the 
Hausdorff-Pompeiu metric, to the superellipse $\C_{p}$. More precisely,
$$
\dH{\mathcal{P}_{p,k},\, \C_{p}} \longrightarrow 0 
 \text{ as } k \to \infty.
$$

Consequently, the sequence of polyhedrons $(\mathcal{Q}_{p,k})_{k\ge 3}$ converges, in the 
Hausdorff-Pompeiu metric, to the superelliptic disk $\D_{p}$:
$$
\dH{\mathcal{Q}_{p,k},\, \D_{p}} \longrightarrow 0 
 \text{ as } k \to \infty.    
$$
\end{theorem}

\begin{proof}
\textbf{Step 1 (from chord to arc).}
Let $x\in \mathcal{P}_{p,k}$ be an arbitrary point. Then $x$ belongs to an edge $\F_t=[P_t,P_{t+1}]$, for some $t\in\{0,1,\ldots,k-1\}$.
Since $P_t,P_{t+1}\in \C_{p}$, we have
$
\inf_{y\in \C_{p}}\|x-y\|
\le
\min\{\|x-P_t\|,\|x-P_{t+1}\|\}
\le
\|P_{t+1}-P_t\|.
$

Moreover, the length of the chord is bounded by the length of the arc:
$
\|P_{t+1}-P_t\|\le \larc{P_tP_{t+1}}.
$

Hence,
$
\inf_{y\in \C_{p}}\|x-y\|
\le
\larc{P_tP_{t+1}}.
$
Taking the supremum over all $x\in \P_{p,k}$ yields
$$
\sup_{x\in P_{p,k}}\inf_{y\in \C_{p}}\|x-y\|
\le
\max_{0\le t\le k-1}\larc{P_tP_{t+1}}.
$$

\smallskip
\textbf{Step 2 (from arc to chord).}
Let $y\in \C_{p}$ be arbitrary. Then $y$ lies on an arc $\arc{P_tP_{t+1}}$, for some $t\in\{0,1,\ldots,k-1\}$.
Since the chord $[P_t,P_{t+1}]$ is contained in the polygon $\P_{p,k}$, we obtain
$
\inf_{x\in \P_{p,k}}\|y-x\|
\le
\min\{\|y-P_t\|,\|y-P_{t+1}\|\}.
$

The distance from $y$ to an endpoint of $[P_t,P_{t+1}]$ is bounded above by the length of the corresponding subarc, and thus by the total arc length:
$
\min\{\|y-P_t\|,\|y-P_{t+1}\|\}
\le
\larc{P_tP_{t+1}}.
$

Therefore,
$
\inf_{x\in \P_{p,k}}\|y-x\|
\le
\larc{P_tP_{t+1}}.
$

Taking the supremum over all $y\in \C_{p}$ gives
$
\sup_{y\in \C_{p}}\inf_{x\in \P_{p,k}}\|y-x\|
\le
\max_{0\le t\le k-1}\larc{P_tP_{t+1}}.
$

\smallskip
\textbf{Conclusion.}
Combining \textbf{Step 1.} and \textbf{Step 2.}, we obtain
$$
\dH{\P_{p,k},\C_{p}}
\le
\max_{0\le t\le k-1}\larc{P_tP_{t+1}}.
$$

Now consider the map
$$
\gamma:[0,2\pi]\to \C_{p},\qquad \gamma(\theta)=\varphi_{p}(\cos\theta,\sin\theta).
$$

The map $\gamma$ is continuous because it is a composition of the continuous maps
$\theta \mapsto (\cos\theta,\sin\theta)$ and $\varphi_{p}$ (see Remark \ref{remark:diffeomorphism}).
Since $\gamma$ is continuous on the compact interval $[0,2\pi]$, it is uniformly continuous. Hence, for every $\varepsilon>0$ there exists $\delta>0$ such that
$
|\theta-\phi|<\delta \ \Longrightarrow\ \|\gamma(\theta)-\gamma(\phi)\|<\varepsilon.
$

Choose $k\in\mathbb{N}^{\ast}$ large enough so that $\frac{2\pi}{k}<\delta$. Then for every $t\in\{0,1,\ldots,k-1\}$, we have
\[
\|P_{t+1}-P_t\|
=
\left\|\gamma\left(\tfrac{2\pi(t+1)}{k}\right)-\gamma\left(\tfrac{2\pi t}{k}\right)\right\|
<\varepsilon.
\]

The superellipse is a chord-arc curve, i.e., there there exists a constant $c>0$ such that $\larc{AB}  \le c\,|AB|$, for any two points $A$ and $B$ on the superellipse. This follows from a general result (see for example  \cite[Theorem 7.9]{book:Pommerenke92}). 
It follows that
\[
\max_{0\le t\le k-1}\larc{P_tP_{t+1}}\longrightarrow 0 \text{ as } k\to\infty.
\]

Therefore,
$
\dH{\P_{p,k},\C_{p}}\longrightarrow 0 \text{ as } k\to\infty,
$
which proves that $(\P_{p,k})_{k\ge3}$ converges to $\C_{p}$ in the Hausdorff-Pompeiu metric.
\end{proof}

\begin{corollary}\label{cor:dense_superellipse}
    The set of all the vertices of the polygons $(\mathcal{P}_{p,k})_{k\ge 3}$, i.e., $$\varphi_p(\mathrm{Tor}(\mathbb{T}))=\bigcup_{k\geq 3}\varphi_p(\mathbb{T}_k),$$ is dense in $\C_p$.
\end{corollary}
\begin{proof}
    It is straightforward by Theorem \ref{convergence of the polygons} and Theorem \ref{th:densitate}.
\end{proof}

\section{Algorithm for the metric projection onto a superellipse} \label{sec:algorithm}

This section uses the same notations as Section \ref{sec:superellipse} and it focuses on creating and studying the following algorithm.

\begin{algorithm} \label{algo:main}
Consider a fixed point \(x^\ast\in\mathbb{R}^2\). 
If $x^*\in\D_p$, then its metric projection onto the superelliptic disk $\mathcal{D}_{p}$ is $x^*$. Otherwise, if  \(x^\ast\in\mathbb{R}^2\setminus \mathcal{D}_{p}\), choose an initial integer \(k\ge 3\) and construct the polygon \(\P_{p,k}\).

\medskip
\textbf{Step 1}
\textbf{(Initial metric projection on \(\P_{p,k}\)).} Compute
$
x^1 = P_{\P_{p,k}}(x^\ast).
$
According to Proposition \ref{prop:partition_R_minus_polyhedron} there are two possibilities:

\begin{itemize}
  \item[(a)]
  \(x^1\in\mathrm{rint}(\F_t)\) for a unique edge \(\F_t=[P_t,P_{t+1}]\).
  \item[(b)]
  \(x^1=P_t\) for some vertex \(P_t\).
\end{itemize}

\medskip
\textbf{Step 2}
\textbf{(Refinement step).}
After $x^1$ is obtained, we consider the $2k$-th roots of unity and the polygon $\P_{p,2k}$, with vertices denoted with $Q_l$, where $l \in\{ 0,1, \ldots, 2k-1\}$. Note that $P_l=Q_{2l}$ for all  $l \in\{ 0,1, \ldots, k-1\}$. 
Again, by Proposition \ref{prop:partition_R_minus_polyhedron} there are two possibilities:

\begin{itemize}
  \item[(a)]
If $x^1\in\mathrm{rint}(\F_t)$, we only need to compute the metric projection of $x^*$ onto the edges $[Q_{2t},Q_{2t+1}]$ and $[Q_{2t+1},Q_{2t+2}]$, and decide which one is at a minimum distance.  Note that indices are understood modulo $2k$. The metric projection of $x^*$ onto the new polygon $\P_{p,2k}$, coincides with the metric projection onto one of those edges. 

\item[(b)]If $x^1=P_t$, then we compute the metric projection of $x^*$ onto the adjacent edges of the vertex $Q_{2t}$ of the polygon $\P_{p,2k}$, namely $[Q_{2t-1},Q_{2t}]$ and $[Q_{2t},Q_{2t+1}]$. Again, we decide which one is at the minimum distance, and thus we obtain the metric projection of $x^*$ onto the new polygon.
\end{itemize}

We denote by $x^2$ the metric projection of $x^*$ onto the polygon $\P_{p,2k}$.

\smallskip
Repeating the refinement step, each time doubling the number of vertices
and testing only the distances between $x^*$ and two edges, we obtain a sequence $(x^n)_{n\geq 1}$
which converges to the metric projection of \(x^\ast\) onto the superellipse $\C_{p}$ (see Theorem \ref{th:algorithm_convergence}).
\end{algorithm}

We start by proving that Algorithm \ref{algo:main} is convergent.

\begin{theorem} \label{th:algorithm_convergence}
    Let \(x^\ast\in\mathbb{R}^2\) be a fixed point. 
    If $x^*\in\D_p$, then its metric projection onto the superelliptic disk $\mathcal{D}_{p}$ is $x^*$. Otherwise, if  \(x^\ast\in\mathbb{R}^2\setminus \mathcal{D}_{p}\), then for any initial integer value of $k\geq 3$, the sequence $(x^n)_{n\geq 1}$, generated by Algorithm \ref{algo:main} converges to the metric projection of \(x^\ast\) onto the superellipse $\C_{p}$.
\end{theorem}
\begin{proof}
    By Theorem \ref{convergence of the polygons} the sequence of polyhedrons $(\mathcal{Q}_{p,k})_{k\ge 3}$ converges in the Hausdorff-Pom\-pe\-iu metric to the superelliptic disk $\D_{p}$. By Corollary \ref{cor:polygonal_convex_bodies} it follows that for every $x^*\in\mathbb{R}^2$ we have $P_{\mathcal{Q}_{p,k}}(x^*)\longrightarrow P_{\D_{p}}(x^*) \text{ as } k\to\infty.$
    
    Thus, for $k$ sufficiently large, the projection of $x^*\in\mathbb{R}^2\setminus \mathcal{Q}_{p,k}$ onto the polyhedron $\mathcal{Q}_{p,k}$ provides an approximation of the metric projection of $x^*$ onto the superelliptic disk $\D_{p}$. 
\end{proof}

In the following example, we will consider a point with a known metric projection onto the superellipse, and we will apply our algorithm (Algorithm \ref{algo:main}) in order to compute its step-by-step error. 

\begin{example}\label{example:algo_main_with_exact_errors}
Consider the superellipse of parameters $a=\sqrt{15}$ and  $b=\sqrt{5}$ and of order $4$:
$$\C_4:=\left\{(x,y)\in\mathbb R^2\mid\ 
\frac{x^{4}}{225}+\frac{y^{4}}{25}=1\right\}.$$

We consider the point $x^0(3,2)\in\C_{4}$. The tangent line to the superellipse in $x^0(3,2)$  is $3/25x+8/25y=1.$ Let $x^\ast(3.75,4)$ be a point on the normal of this tangent (see Figure \ref{fig:tangent}). Hence we know that $P_{\C_{4}}(x^\ast)=x^0.$

We begin applying Algorithm \ref{algo:main}, by choosing an initial integer \(k=6\) and constructing the polygon \(\P_{4,6}\) (see Figure \ref{fig:polygon_4_6}). Before we start, note that due to the complexity of the underlying analytical expressions, all subsequent computations are carried out using the computer algebra system Maple. Numerical quantities are evaluated in high precision arithmetic with \texttt{Digits := 50}, and final reported results are rounded to 10 decimal places.

\begin{figure}[H]
    \centering
    \begin{subfigure}{0.4\textwidth}
        \centering
        \includegraphics[width=\linewidth, trim=0 2.8cm 0 2.8cm,
  clip]{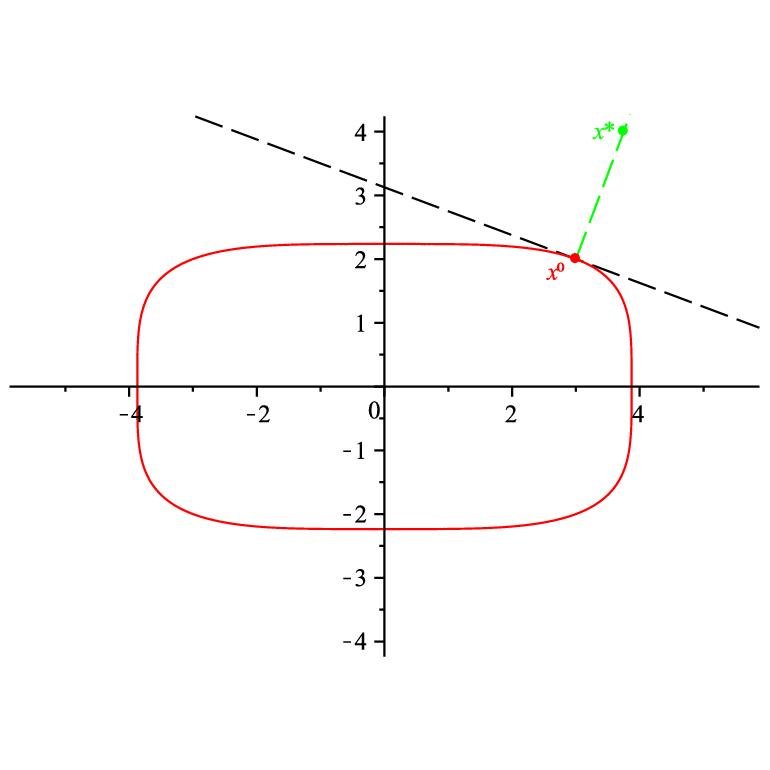}
        \caption{Construction of the point $x^\ast$}
        \label{fig:tangent}
    \end{subfigure}
    \qquad
    \begin{subfigure}{0.4\textwidth}
        \centering
        \includegraphics[width=\linewidth, trim=0 2.8cm 0 2.8cm,
  clip]{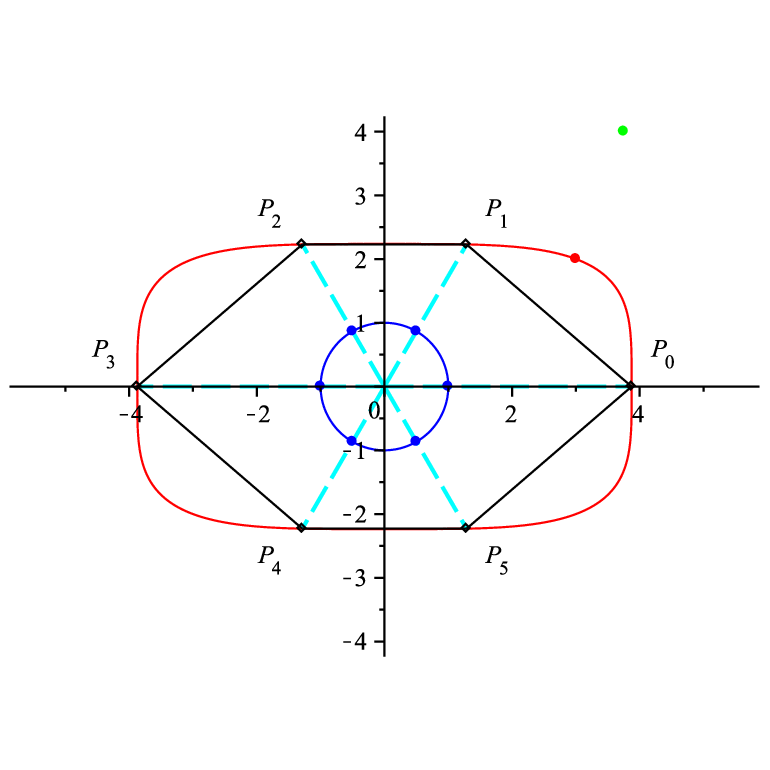}
        \caption{The polygon \(\P_{4,6}\)}
        \label{fig:polygon_4_6}
    \end{subfigure}
    \caption{Setting up Algorithm \ref{algo:main} for Example \ref{example:algo_main_with_exact_errors}}
    \label{fig:example_exact_error_setup}
\end{figure}

\smallskip
\textbf{Step 1}
\textbf{(Initial metric projection on \(\P_{4,6}\)).} Drawing the rectangular and triangular regions (see Remark \ref{remark:region_inequalities} and Figure \ref{fig:ex-regions}), note that $x^*$ belongs to the first rectangular region $P^{-1}_{\mathcal{Q}_{4,6}}(\mathrm{rint} (\F_0))$, hence $x^1 = P_{\P_{4,6}}(x^\ast)\in\mathrm{rint} (\F_0)$. Computationally, indeed we have the following numerical approximation $$x^1\approx(1.8242639597, 1.7661042090)\in\mathrm{rint} (\F_0)$$ (see Figure \ref{fig:ex-first_approx}), with an absolute error of 
$\|x^0-x^1\|\approx 1.1987754074.$

\begin{figure}[H]
    \centering
    \begin{subfigure}{0.4\textwidth}
        \centering
        \includegraphics[width=\linewidth, trim=0 2.8cm 0 2.8cm,
  clip]{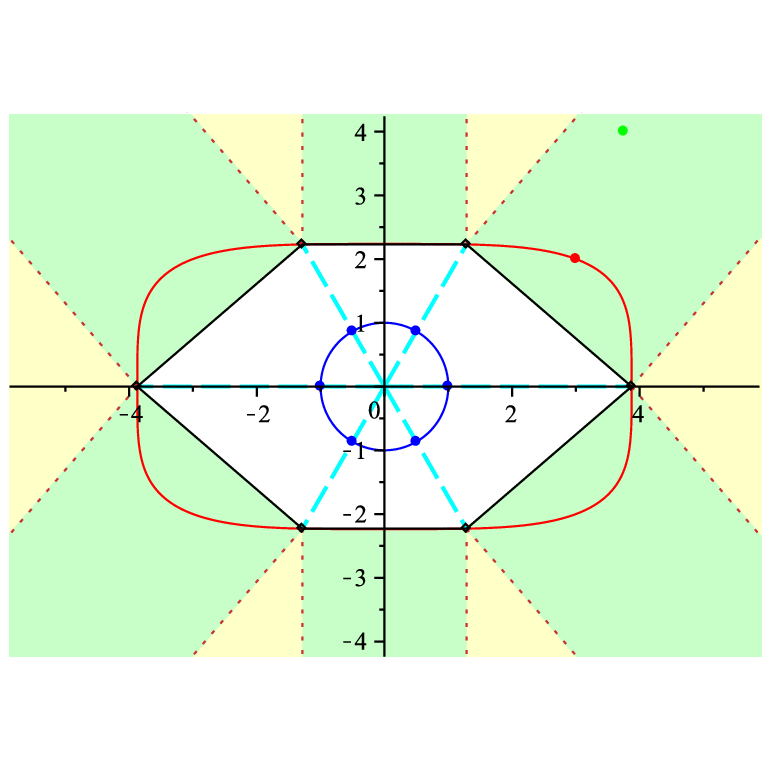}
        \caption{Rectangular and triangular regions}
        \label{fig:ex-regions}
    \end{subfigure}
    \qquad
    \begin{subfigure}{0.4\textwidth}
        \centering
        \includegraphics[width=\linewidth, trim=0 2.8cm 0 2.8cm,
  clip]{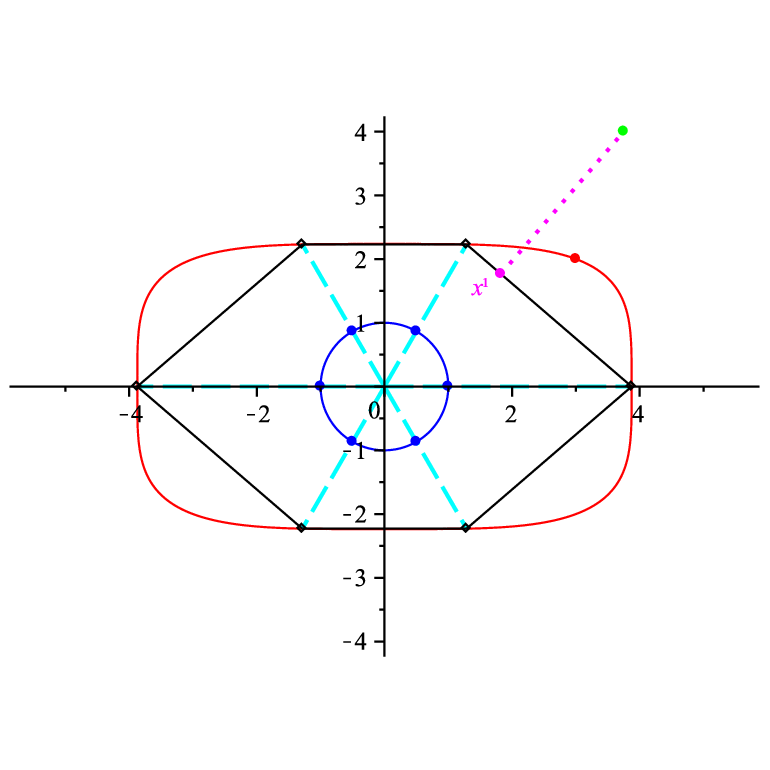}
        \caption{The point $x^1$}
        \label{fig:ex-first_approx}
    \end{subfigure}
    \caption{Step 1 of Algorithm \ref{algo:main} applied on Example \ref{example:algo_main_with_exact_errors}}
    \label{fig:example_exact_error_Step1}
\end{figure}

\smallskip
\textbf{Step 2}
\textbf{(Refinement step).} We consider the 12 twelfth roots of unity and the polygon $\P_{4,12}$, with vertices denotes by $Q_0, Q_1,\ldots Q_{11}$. We only need to consider the edges $[Q_0,Q_1]$ and $[Q_1,Q_2]$ (see Figure \ref{fig:ex-refinement}). By taking the metric projections of $x^*$ onto this edges, we note that both metric projections are given by the point $Q_1$. Visually this is clear from Figure \ref{fig:ex-refinement_regions}, where we note that $x^*$ lies in the triangular region given by the vertex $Q_1$, i.e., $x^*\in P^{-1}_{\mathcal{Q}_{4,12}}(Q_1)$. Therefore, we have obtained the next approximation point, estimated to be 
$$x^2=Q_1\approx(3.2567778122, 1.8803015465)$$
(see Figure \ref{fig:ex-second_approx}), with an absolute error of 
$$\|x^0-x^2\|\approx 0.2833064852.$$

\smallskip
On the next refinement step, we move onto the polygon $\P_{4,24}$ (see Figure \ref{fig:ex-2nd_refinement}), where $x^2$ is a vertex, and via our algorithm we obtain the  estimate 
$$x^3\approx (3.1716810745, 1.9037785768),$$
with an absolute error of 
$$\|x^0-x^3\|\approx 0.1968068942.$$

\begin{figure}[H]
    \centering
    \begin{subfigure}{0.4\textwidth}
        \centering
        \includegraphics[width=\linewidth, trim=0 2.8cm 0 2.8cm,
  clip]{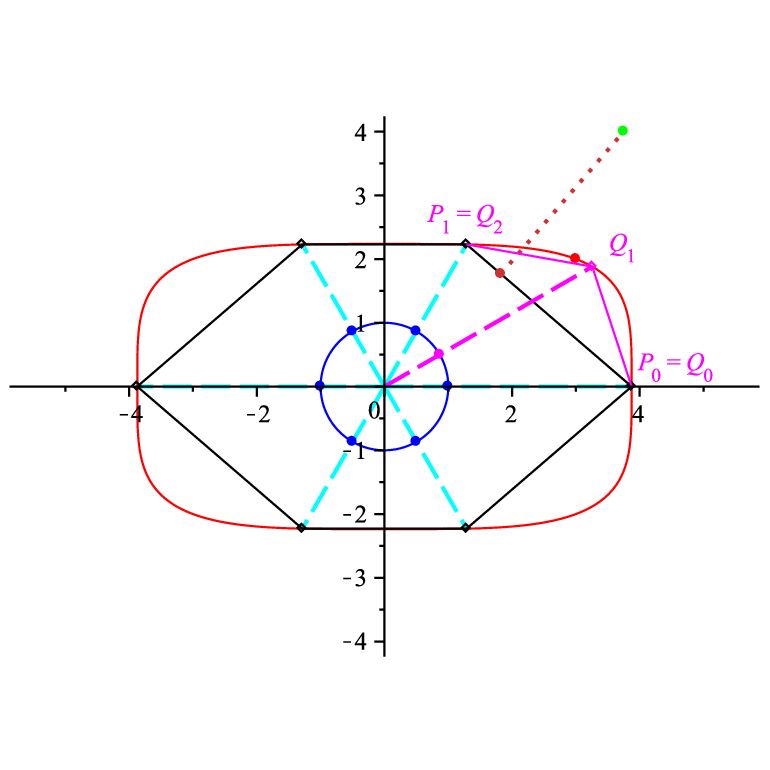}
        \caption{Refinement via the polygon $\P_{4,12}$}
        \label{fig:ex-refinement}
    \end{subfigure}
    \qquad
    \begin{subfigure}{0.4\textwidth}
        \centering
        \includegraphics[width=\linewidth, trim=0 2.8cm 0 2.8cm,
  clip]{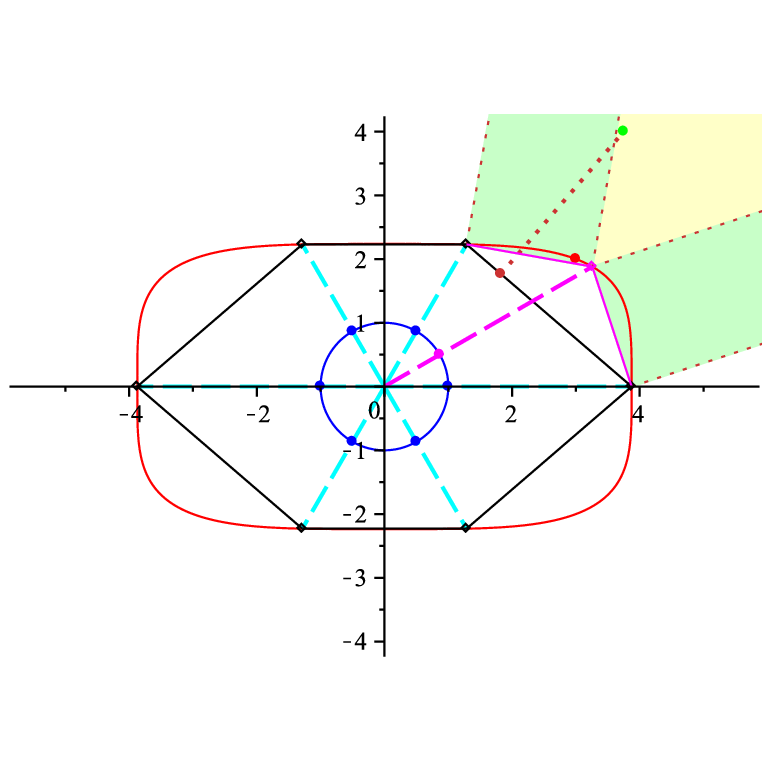}
        \caption{Refinement regions}
        \label{fig:ex-refinement_regions}
    \end{subfigure}\\[0.2cm]
    \begin{subfigure}{0.4\textwidth}
        \centering
        \includegraphics[width=\linewidth, trim=0 2.8cm 0 2.8cm,
  clip]{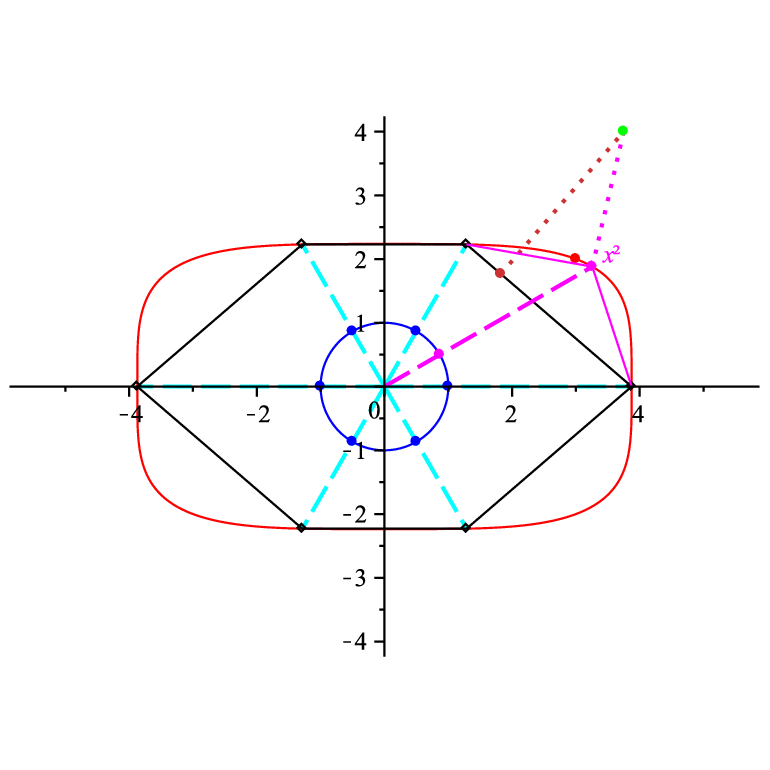}
        \caption{The point $x^2$}
        \label{fig:ex-second_approx}
    \end{subfigure}
    \qquad
    \begin{subfigure}{0.4\textwidth}
        \centering
        \includegraphics[width=\linewidth, trim=0 2.8cm 0 2.8cm,
  clip]{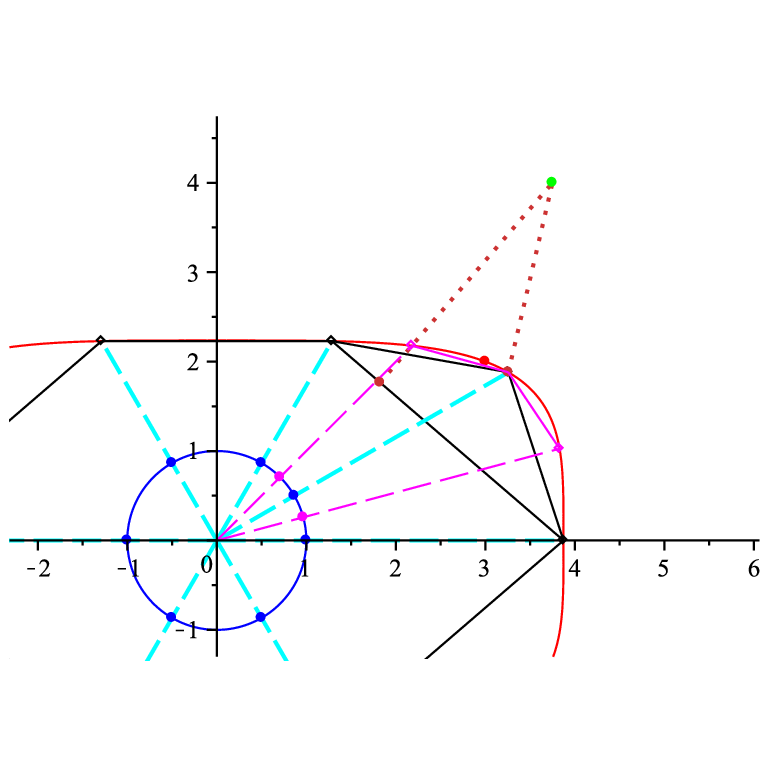}
        \caption{Refinement via the polygon $\P_{4,24}$}
        \label{fig:ex-2nd_refinement}
    \end{subfigure}
    \caption{Step 2 of Algorithm \ref{algo:main} applied on Example \ref{example:algo_main_with_exact_errors}}
    \label{fig:example_exact_error2}
\end{figure}

\smallskip
Repeating the refinement step, we would obtain a sequence $(x^n)_{n\geq 1}$
which converges to the metric projection of \(x^\ast(3.75,4)\) onto the superellipse $\C_{4}$, namely $x^0(3,2)$. In Table \ref{table:example_errors} we present the first 10 iterations of our algorithm, together with the absolute errors obtained.

\begin{table}[H]
\centering{
\renewcommand{\arraystretch}{1.3}
\begin{tabularx}{\textwidth}{||Y|Y|Y||}
\hline
\hline
Point & Coordinates & {Absolute error} \\
\hline
\hline
$x^1$ & $(1.8242639597, 1.7661042090)$ & $1.1987754074$\\
\hline
$x^2$ & $(3.2567778122, 1.8803015465)$ & $0.2833064852$\\
\hline
$x^3$ & $(3.1716810745, 1.9037785768)$ & $0.1968068942$\\
\hline
$x^4$ & $(2.9806294820, 1.9857818721)$ & $0.0240285690$\\
\hline
$x^5$ & $(2.9956465266, 2.0016270159)$ & $0.0046475704$\\
\hline
$x^6$ & $(2.9956465266, 2.0016270159)$ & $0.0046475704$\\
\hline
$x^7$ & $(2.9956465266, 2.0016270159)$ & $0.0046475704$\\
\hline
$x^8$ & $(2.9956465266, 2.0016270159)$ & $0.0046475704$\\
\hline
$x^9$ & $(2.9956984396, 2.0016074211)$ & $0.0045920828$\\
\hline
$x^{10}$&$(3.0000939930, 1.9999594588)$ & $0.0001023635$\\
\hline
\hline
\end{tabularx}}
\caption{First ten metric projection approximations}
\label{table:example_errors}
\end{table}

Note that our sequence is non-injective, for example, in Table \ref{table:example_errors}, we have that $x^5=x^6=x^7=x^8$. This happens, because the approximation $x^5$ became so close to $x^0(3,2)$, that the refinement step took 4 additional iterations, i.e. to double the number of vertices of the polygon; in order to catch up to the distance between $x^0$ and $x^5$.
\end{example}

\phantomsection

\end{document}